\documentclass[11pt]{amsart}

\usepackage[utf8]{inputenc}
\usepackage{amsmath,amsthm,amssymb}
\usepackage{indentfirst}
\usepackage{url}
\usepackage{tikz-cd}
\usepackage[colorlinks=true]{hyperref}
\usepackage{colonequals}
\usepackage{graphicx}
\usepackage{enumitem}
\usepackage{soul}
\usepackage{stmaryrd}
\usepackage{mathrsfs,mathtools}
\usepackage{tgpagella}
\usepackage{bbm}
\usepackage{subcaption}
\usepackage[alphabetic]{amsrefs}
\usepackage{xcolor}

\usepackage[margin=1in]{geometry}

\newtheorem{thm}[subsection]{Theorem}
\newtheorem{lem}[subsection]{Lemma}
\newtheorem{conj}[subsection]{Conjecture}
\newtheorem{prop}[subsection]{Proposition}

\newtheorem{cor}[subsection]{Corollary}
{
\theoremstyle{definition}
\newtheorem{remx}[subsection]{Remark}

\newtheorem{defnx}[subsection]{Definition}
\newtheorem{examplex}[subsection]{Example}
}
\newenvironment{rem}
{\pushQED{\qed}\remx}
{\popQED\endremx}

\newenvironment{defn}
{\pushQED{\qed}\defnx}
{\popQED\enddefnx}
\newenvironment{example}
{\pushQED{\qed}\examplex}
{\popQED\endexamplex}

\newcommand{\NN}{\mathbb N}
\newcommand{\ZZ}{\mathbb Z}
\newcommand{\RR}{\mathbb R}
\newcommand{\CC}{\mathbb C}

\newcommand{\del}{\partial}
\DeclareMathOperator{\PD}{{PD}}
\newcommand{\sort}{\mathrm{sort}}
\newcommand{\os}{\mathrm{os}}

\begin{document}
\title{Double orthodontia formulas and Lascoux positivity
}
\author{Linus Setiabrata}
\address{Linus Setiabrata, Department of Mathematics, Massachusetts Institute of Technology, Cambridge, MA, 02139. \newline\textup{setia@mit.edu}
}

\author{Avery St.~Dizier}
\address{Avery St.~Dizier, Department of Mathematics, Michigan State University, East Lansing, MI 48824. \newline\textup{stdizier@msu.edu}
}

\begin{abstract}
We give a new formula for double Grothendieck polynomials based on Magyar's orthodontia algorithm for diagrams. Our formula implies a similar formula for double Schubert polynomials $\mathfrak S_w(\mathbf x;\mathbf y)$. We also prove a curious positivity result: for vexillary permutations $w\in S_n$, the polynomial $x_1^n\dots x_n^n \mathfrak S_w(x_n^{-1}, \dots, x_1^{-1}; 1,\dots,1)$ is a graded nonnegative sum of Lascoux polynomials. We conjecture that this positivity result holds for all $w\in S_n$. This conjecture would follow from a problem of independent interest regarding Lascoux positivity of certain products of Lascoux polynomials.
\end{abstract}

\maketitle

\section{Introduction}
\subsection*{Orthodontia and flagged Weyl modules}
Schubert polynomials $\mathfrak S_w(x_1, \dots, x_n)$, introduced by Lascoux and Sch\"utzenberger \cite{ls82}, are distinguished representatives of Schubert classes in the cohomology of the flag variety of $\CC^n$. These polynomials have very rich combinatorial structure \cite{bb93,bjs93,fk96,km05,lls21,hmms22} and play a central role in algebraic combinatorics.

Schubert polynomials $\mathfrak S_w(\mathbf x)$ are indexed by permutations $w \in S_n$. The Rothe diagram $D(w)$ of $w\in S_n$ encodes a plethora of information about $\mathfrak S_w(\mathbf x)$. More generally, for any \%-avoiding diagram $D$, there is a representation of the group $B$ of invertible upper triangular matrices called the \emph{flagged Weyl module} $\mathcal M_D$. When $D = D(w)$ is a Rothe diagram, the dual character $\chi_D$ of the representation $\mathcal M_D$ is equal to the Schubert polynomial $\mathfrak S_w(\mathbf x)$ \cite{kp87,kp04}. The study of dual characters $\chi_D$ as a whole has shed light on Schubert and related polynomials \cite{fms18,hmss24}.

Using a geometric interpretation of $\mathcal M_D$ as the space of sections of a certain line bundle on a variety, Magyar \cite{magyar98} showed that $\chi_D$ is given by the formula
\[
\chi_D = \omega_1^{k_1}\dots\omega_n^{k_n}\pi_{i_1}(\omega_{i_1}^{m_1}\pi_{i_2}(\omega_{i_2}^{m_2}(\dots\pi_{i_\ell}(\omega_{i_\ell}^{m_\ell})\dots))).
\]
Here, $\omega_i = x_1\dots x_i$ is a fundamental weight, $\pi_i = \del_i x_i$ is a Demazure operator, and
\[
\mathbf k(D) = (k_1,\dots, k_n), \qquad \mathbf i(D) = (i_1, \dots, i_\ell), \qquad \mathbf m(D) = (m_1, \dots, m_\ell)
\]
is combinatorial data associated to the \emph{orthodontic sequence} of $D$, which builds a \%-avoiding diagram from ``smaller'' \%-avoiding diagrams.

\subsection*{Doubled orthodontia}
Double Grothendieck polynomials $\mathfrak G_w(x_1, \dots, x_n; y_1, \dots, y_n)$, introduced by Lascoux \cite{lascoux90}, are distinguished representatives of Schubert classes in the equivariant $K$-theory of the flag variety. Schubert polynomials can be obtained from double Grothendieck polynomials by setting $y_i\mapsto 0$ (corresponding to forgetting equivariance) and taking the lowest degree part (corresponding to taking the associated graded in $K$-theory). There is considerable interest in extending our understanding of Schubert polynomials to double Grothendieck polynomials and their various specializations \cite{fk94, km04,sam11, weigandt21, lls23, ccmm23,bfhtw23, psw24, hmss24}.

There is no known $K$-theoretic analogue of $\mathcal M_D$. Despite this, we can combinatorially extend Magyar's formula to double Grothendieck polynomials:
\begin{thm}
\label{thm:double-groth-master}
Let $D$ be a \%-avoiding diagram with double orthodontic sequence $\mathbf K, \mathbf i, \mathbf j, \mathbf M$. Define
\begin{equation}
\label{eqn:double-groth-master}
\mathscr G_D(\mathbf x;\mathbf y)\colonequals \overline\omega_1^{K_1}\overline\omega_2^{K_2}\dots\overline\omega_n^{K_n}\overline\pi_{i_1,j_1}(\overline\omega_{i_1}^{M_1}\overline\pi_{i_2,j_2}(\overline\omega_{i_2}^{M_2}\dots\overline\pi_{i_\ell,j_\ell}(\overline\omega_{i_\ell}^{M_\ell})\dots)).
\end{equation}
When $D = D(w)$ is the Rothe diagram of a permutation, then $\mathscr G_D(\mathbf x; \mathbf y) = \mathfrak G_w(\mathbf x; \mathbf y)$.
\end{thm}
The operators $\overline\omega_i^M$ and $\overline\pi_{i,j}$ are defined in Equations~\eqref{eqn:pi-ij-operators} and~\eqref{eqn:omega-iM-operators}.

Double Schubert polynomials $\mathfrak S_w(\mathbf x; \mathbf y)$ are equal to the lowest degree part of $\mathfrak G_w(\mathbf x; -\mathbf y)$, and Theorem~\ref{thm:double-groth-master} implies a similar formula for double Schubert polynomials (Corollary~\ref{cor:double-schub-master}). Theorem~\ref{thm:double-groth-master} also extends our previous joint work with M\'esz\'aros \cite[Thm 1.1]{mss22}, which gave a similar formula for ordinary Grothendieck polynomials.

As is the case with Magyar's formula for $\chi_D$, and in contrast to the usual degree-\emph{decreasing} recurrence defining Grothendieck polynomials via divided difference operators, the formula~\eqref{eqn:double-groth-master} is degree-\emph{increasing}. This makes the formula~\eqref{eqn:double-groth-master} well suited for some combinatorial applications.

%we plan to explore in future work(???? --places where ordinary orthodontia is useful and where double grothendieck analogues are open: Kohnert, peelable, and zero-one).}

\subsection*{Lascoux positivity}
When $D = D(\alpha)$ is the skyline diagram of a composition $\alpha\in\NN^n$, the dual character $\chi_D$ is the \emph{key polynomial} $\kappa_\alpha(x_1, \dots, x_n)$ \cite{rs95}. Key polynomials are minimal among the dual characters $\chi_D$ of \%-avoiding diagrams: every dual character, and in particular every Schubert polynomial, is a nonnegative sum of key polynomials \cite{rs98}.

The Lascoux polynomials $\mathfrak L_\alpha(x_1, \dots, x_n)$, indexed by compositions $\alpha\in\NN^n$, are inhomogeneous polynomials which often play for key polynomials the same role as Grothendieck polynomials play for Schubert polynomials. One inhomogeneous extension of the key positivity \cite{rs98} of Schubert polynomials is the Grothendieck-to-Lascoux expansion \cite{ry21,sy23}. One application of our formula~\eqref{eqn:double-groth-master} is another inhomogeneous extension of this key positivity:

\begin{thm}
\label{thm:lascoux-positivity-main}
Let $D\subseteq[n]\times[m]$ be a diagram whose columns are ordered by inclusion. Let $\mathscr S_D(\mathbf x; \mathbf y)$ be the lowest degree part of $\mathscr G_D(\mathbf x; \mathbf y)$. Then, the polynomial
\[
x_1^m \dots x_n^m \mathscr S_D(x_n^{-1}, \dots, x_1^{-1}; -1, \dots, -1)
\]
is a graded nonnegative sum of Lascoux polynomials $\mathfrak L_\alpha(x_1, \dots, x_n)$.
\end{thm}
Theorem~\ref{thm:lascoux-positivity-main} can be modified to incorporate $\beta$-Lascoux polynomials: the $\beta$-Lascoux expansion of $x_1^m\dots x_n^m \mathscr S_D(x_n^{-1}, \dots, x_1^{-1}; \beta,\dots,\beta)$ is $\ZZ_{\geq 0}[\beta]$-nonnegative.

Theorem~\ref{thm:lascoux-positivity-main} and Corollary~\ref{cor:double-schub-master} together imply the following corollary.
\begin{cor}
\label{cor:lascoux-positivity-schub}
Let $w\in S_n$ be a vexillary permutation, and write $\mathfrak S_w(x_1, \dots, x_n; y_1, \dots, y_n)$ for the double Schubert polynomial. Then the polynomial
\[
x_1^n \dots x_n^n \mathfrak S_w(x_n^{-1}, \dots, x_1^{-1}; 1, \dots, 1)
\]
is a graded nonnegative sum of Lascoux polynomials $\mathfrak L_\alpha(x_1, \dots, x_n)$.
\end{cor}

We conjecture that the extra assumption on the diagram $D$ in Theorem~\ref{thm:lascoux-positivity-main} is unnecessary.
\begin{conj}
\label{conj:lascoux-positivity-general}
For any \%-avoiding diagram $D\subseteq[n]\times[m]$, the polynomial
\[
x_1^m \dots x_n^m \mathscr S_D(x_n^{-1}, \dots, x_1^{-1}; -1, \dots, -1)
\]
is a graded nonnegative sum of Lascoux polynomials $\mathfrak L_\alpha(x_1, \dots, x_n)$. In particular, the same holds for
\[
x_1^n\dots x_n^n\mathfrak S_w(x_n^{-1}, \dots, x_1^{-1}; 1, \dots, 1),
\]
for any $w \in S_n$.
\end{conj}
In Proposition~\ref{prop:mult-implies-main}, we show that Conjecture~\ref{conj:lascoux-positivity-general} would follow from the following conjecture of independent interest:
\begin{conj}
\label{conj:mult-pos}
For any $\alpha\in\NN^n$ and $i\in[n]$, the product $x_1\dots x_i(1-x_{i+1})\dots(1-x_n)\mathfrak L_\alpha$ is a graded nonnegative linear combination of Lascoux polynomials.
\end{conj}

\subsection*{Outline of the paper}
Our proof of Theorem~\ref{thm:double-groth-master} is based on a refinement of the ideas in \cite{mss22}: we induct on a partial order on $S_n$ called the \emph{orthodontic sort order} (Definition~\ref{defn:os}), and a key step is to give a precise description of the first few terms of the orthodontic sequences of \emph{sorted permutations} (Theorem~\ref{thm:os2-orth}). Although the argument is narratively similar to that in \cite{mss22}, we need different methods.

Our proof of Theorem~\ref{thm:lascoux-positivity-main} examines the orthodontic sequence of diagrams $D$ whose columns are totally ordered by inclusion (Corollary~\ref{cor:omegas-then-pis}) and studies the behavior of the operators $\omega_i^M$ and $\pi_{i,j}$ when specializing $y_j\mapsto -1$ and intertwining with the operator $r_{m,n}$ of \cite[Defn 1.1] {yu23} defined on polynomials by
\[
f(x_1, \dots, x_n)\mapsto x_1^m \dots x_n^m f(x_n^{-1}, \dots, x_1^{-1}).
\]
\section*{Acknowledgments}
We thank David Anderson, Sergey Fomin, Peter Magyar, Karola M\'esz\'aros, Alexander Yong, and Tianyi Yu for many helpful conversations and useful comments which improved our exposition. We also thank the anonymous referee for a careful reading of the manuscript.
\section{Background}
\subsection*{Conventions}
We write permutations $w \in S_n$ in one-line notation. For $j\in[n-1]$, the notation $s_j$ denotes the adjacent transposition in $S_n$ which swaps $j$ and $j+1$. For $w\in S_n$, let $\ell(w)$ denote the length of $w$. Permutations act on the right: $ws_j$ is equal to $w$ with $w(j)$ and $w(j+1)$ swapped.
\subsection*{Difference operators and families of polynomials}
For $i \in [n-1]$, define the \emph{divided difference operator} $\del_i\colon \RR[x_1, \dots, x_n]\to \RR[x_1, \dots, x_n]$ by the formula
\[
\del_i(f) \colonequals \frac{f - s_i f}{x_i - x_{i+1}}.
\]
The \emph{isobaric divided difference operators} $\overline\del_i$, \emph{Demazure operators} $\pi_i$, and \emph{Demazure--Lascoux operators} $\overline\pi_i$ are defined on $\RR[x_1, \dots, x_n]$ by the formulas
\begin{align*}
\overline\del_i(f) &\colonequals \del_i((1 - x_{i+1})f),\\
\pi_i(f) &\colonequals \del_i(x_if),\\
\overline\pi_i(f) &\colonequals \overline\del_i(x_if).
\end{align*}
\begin{lem}
\label{lem:varphi-braid}
Let $\varphi_i$ denote any of $\del_i$, $\overline\del_i$, $\pi_i$, or $\overline\pi_i$. Then
\[
\varphi_i \varphi_{i+1}\varphi_i = \varphi_{i+1}\varphi_i\varphi_{i+1} \qquad\textup{ and } \qquad \varphi_i\varphi_j = \varphi_j\varphi_i\textup{ if } |i-j| \geq 2.
\]
\end{lem}
The \emph{double Grothendieck polynomial} $\mathfrak G_w(\mathbf x; \mathbf y)$ of $w\in S_n$ is defined recursively on the weak Bruhat order, starting from the longest permutation $w_0 \in S_n$. The polynomial $\mathfrak G_w(\mathbf x; \mathbf y)$ is defined by
\[
\mathfrak G_w(\mathbf x; \mathbf y) = \begin{cases} \prod_{i+j\leq n}(x_i + y_j - x_iy_j) &\textup{ if } w = w_0,\\ \overline\del_i\mathfrak G_{ws_i}(\mathbf x; \mathbf y)&\textup{ if } \ell(w) < \ell(ws_i).\end{cases}
\]
Lemma~\ref{lem:varphi-braid} guarantees that double Grothendieck polynomials are well-defined. The \emph{double Schubert polynomial} $\mathfrak S_w(\mathbf x; \mathbf y)$ is defined to be the lowest degree part of $\mathfrak G_w(\mathbf x; -\mathbf y)$. In particular,
\[
\mathfrak S_w(\mathbf x; \mathbf y) = \begin{cases} \prod_{i+j\leq n}(x_i - y_j) &\textup{ if } w = w_0,\\ \del_i\mathfrak S_{ws_i}(\mathbf x; \mathbf y) &\textup{ if } \ell(w) < \ell(ws_i).\end{cases}
\]
The ordinary \emph{Grothendieck polynomials} $\mathfrak G_w(\mathbf x)$ and ordinary \emph{Schubert polynomials} $\mathfrak S_w(\mathbf x)$ are defined to be the specializations $\mathfrak G_w(\mathbf x; \mathbf 0)$ and $\mathfrak S_w(\mathbf x;\mathbf 0)$ of their doubled counterparts. In particular,
\[
\mathfrak G_w(\mathbf x)\! =\! \begin{cases} x_1^{n-1} x_2^{n-2} \dots x_{n-1} &\textup{ if } w = w_0,\\ \overline\del_i\mathfrak G_{ws_i}(\mathbf x) &\textup{ if } \ell(w) < \ell(ws_i),\end{cases} \hspace{0.3em} \textup{ and } \hspace{0.3em} \mathfrak S_w(\mathbf x) \!=\! \begin{cases} x_1^{n-1} x_2^{n-2} \dots x_{n-1} &\textup{ if } w = w_0,\\ \del_i\mathfrak S_{ws_i}(\mathbf x) &\textup{ if } \ell(w) < \ell(ws_i).\end{cases}
\]
The symmetric group $S_n$ acts on compositions via permuting coordinates. The \emph{Lascoux polynomial} $\mathfrak L_\alpha$ of a composition $\alpha \in \NN^n$ is defined recursively by
\[
\mathfrak L_\alpha(\mathbf x) = \begin{cases} x_1^{\alpha_1}\dots x_n^{\alpha_n}&\textup{ if } \alpha_1 \geq \dots \geq \alpha_n\\ \overline\pi_i\mathfrak L_{\alpha\cdot s_i}(\mathbf x) &\textup{ if } \alpha_i < \alpha_{i+1}.\end{cases}
\]
Lemma~\ref{lem:varphi-braid} guarantees that Lascoux polynomials are well-defined. The \emph{key polynomial} $\kappa_\alpha(\mathbf x)$ is defined to be the lowest degree part of $\mathfrak L_\alpha(\mathbf x)$. In particular,
\[
\kappa_\alpha(\mathbf x) = \begin{cases} x_1^{\alpha_1}\dots x_n^{\alpha_n}&\textup{ if } \alpha_1 \geq \dots \geq \alpha_n\\ \pi_i\kappa_{\alpha\cdot s_i}(\mathbf x) &\textup{ if } \alpha_i < \alpha_{i+1}.\end{cases}
\]
\subsection*{Pipe dreams}
A \emph{pipe dream} is a filling of a triangular grid $\{(i,j)\in[n]\times[n]\colon i + j \leq n\}$ with crossing tiles \includegraphics[scale=0.5]{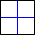} and bump tiles \includegraphics[scale=0.5]{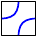}. By placing half-bump tiles \includegraphics[scale=0.5]{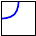} at $\{(i,j)\in[n]\times[n]\colon i+j = n+1\}$, a pipe dream forms a network of $n$ pipes running from the top edge of the grid to the left edge (see Figure~\ref{fig:1423-pd}).

For any pipe dream $P$, there is an associated permutation $\del(P) \in S_n$ given by labeling the pipes $1$ through $n$ along the top edge, tracing the pipes and ignoring any crossings between any pair of pipes which have already crossed, i.e.\ replacing redundant crossing tiles with bump tiles; reading the labels of the pipes along the left edge from top to bottom gives a string of numbers $(\del(P))(1), (\del(P))(2), \dots, (\del(P))(n)$ that defines $\del(P) \in S_n$. (The permutation $\del(P)$ is the \emph{Demazure product} of the transpositions $s_i\in S_n$ corresponding to antidiagonals on which the crosses sit, reading right to left, starting from the top row; cf.\ \cite[Ex 5.1]{km04}, \cite[\S 6.1]{weigandt21}. Our convention agrees with the original definition \cite{bb93}: the pipe at the $i$\textsuperscript{th} row is connected to the $(\del(P))(i)$\textsuperscript{th} column.)

Identify $P$ with the set $\{(i,j) \colon P\textup{ has } \includegraphics[scale=0.5]{images/crossing-tile} \textup{ at } (i,j)\}$ of crosses of $P$, and denote by $P^{(i)}$ the pipe of $P$ entering in the $i$\textsuperscript{th} column. %\note{traverse means follow path as if it were reduced... you might traverse ``cross tiles, but not horizontally/vertically''...}

\begin{example}
\label{ex:1423-pd}
All five pipe dreams $P$ such that $\del(P) = 1423$ are displayed in Figure~\ref{fig:1423-pd}. Redundant crossing tiles \includegraphics[scale=0.5]{images/crossing-tile} are highlighted in red.
\begin{figure}[ht]
\centering
\includegraphics[scale=0.8]{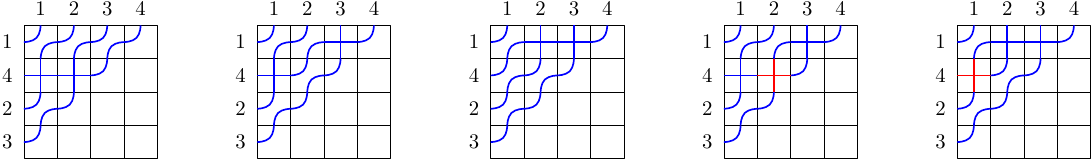}
\caption{The five pipe dreams in $\PD(1423)$.}
\label{fig:1423-pd}
\end{figure}
\end{example}

\begin{thm}[{\cite[Thm 6.1]{weigandt21}; see also \cite[Cor 5.4]{km04}, \cite[Thm 2.3]{fk94}}]
The double Grothendieck polynomial of $w\in S_n$ is given by the formula
\[
\mathfrak G_w(\mathbf x; \mathbf y) = \sum_{P \in \PD(w)}(-1)^{|P| - \ell(w)} \prod_{(i,j) \in P} (x_i + y_j - x_iy_j).
\]
\end{thm}

\subsection*{Orthodontic sequence}
A \emph{diagram} is defined to be a subset $D\subseteq [n]\times[m]$. View $D = (D_1, \dots, D_m)$ as a subset of an $n\times m$ grid, where $D_j\colonequals \{i \in[n]\colon (i,j)\in D\}$ encodes the $j$\textsuperscript{th} column: an element $i \in D_j$ corresponds to a box in row $i$ and column $j$.

\begin{defn}
Let $w \in S_n$. The Rothe diagram $D(w)$ is defined to be
\[
D(w) = \{(i,j)\in[n]\times[n]\colon i < w^{-1}(j) \textup{ and } j < w(i)\}.\qedhere
\]
\end{defn}

\begin{example}
The Rothe diagram $D(31542)$ consists of the purple squares in Figure~\ref{fig:31542-rothe}.
\begin{figure}[ht]
\centering
\includegraphics{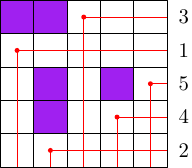}
\caption{The Rothe diagram for $w = 31542$.}
\label{fig:31542-rothe}
\end{figure}
\end{example}

\begin{defn}[\cite{rs98}]
A diagram $D$ is called \emph{\%-avoiding} when $i_2 \in D_{j_1}$ and $i_1 \in D_{j_2}$ for $i_1 < i_2$ and $j_1 < j_2$ implies $i_1 \in D_{j_1}$ or $i_2 \in D_{j_2}$.
\end{defn}

Equivalently, a diagram is \%-avoiding if it does not have a pair of rows and a pair of columns to which its restriction looks like the configuration in Figure \ref{fig:percentage_avoiding}.

\begin{figure}[ht]
	\begin{center}
		\includegraphics{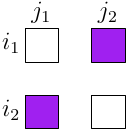}
	\end{center}
	\caption{A forbidden configuration in a \%-avoiding diagram.}
	\label{fig:percentage_avoiding}
\end{figure}

Given a \%-avoiding diagram $D$, we describe an algorithm to compute its \emph{double orthodontic sequence}
\[
\mathbf K(D) = (K_1, \dots, K_n), \hspace{0.8em} \mathbf i(D) = (i_1, \dots, i_\ell), \hspace{0.8em} \mathbf j(D) = (j_1, \dots, j_\ell), \hspace{0.8em} \mathbf M(D) = (M_1, \dots, M_\ell).
\]

Begin by setting $K_i \colonequals \{j\colon D_j = [i]\}$, and let $D_-$ denote the diagram obtained by replacing any such columns with the empty column. If $D_-$ is empty, then $\mathbf i = \mathbf j = \mathbf M$ is the empty vector and the algorithm terminates.

An \emph{orthodontic step} applied to $D_-$ outputs a diagram $(D_-)''$ which we now describe. A \emph{missing tooth} of a column $C\subseteq[n]$ is an integer $i\in[n-1]$ so that $i\not\in C$ and $i+1\in C$; any nonempty column of $D_-$ has a missing tooth. Set $i_1$ to be the smallest missing tooth of the leftmost nonempty column $(D_-)_j$ of $D_-$. Set $j_1$ to be $j - \#\{a\leq i_1\colon a\not\in (D_-)_j\}$. Now swap rows $i_1$ and $i_1+1$ to get a diagram $(D_-)'$. Any column of $(D_-)'$ which is a standard interval is necessarily equal to $[i_1]$; set $M_1 = \{j\colon (D_-)'_j = [i_1]\}$. Let $(D_-)''$ denote the diagram obtained from $(D_-)'$ by removing all columns equal to $[i_1]$. 

\begin{lem}[{\cite[Prop 12]{rs98}, see also \cite[pg.\ 12]{magyar98}}]
Any \%-avoiding diagram $D$ will be the empty diagram after sufficiently many orthodontic steps.
\end{lem}

Repeatedly apply orthodontic steps to the resulting diagrams $(D_-)''$, keeping track of the data $\mathbf i, \mathbf j, \mathbf M$ at each step, until the diagram is empty. These diagrams form the double orthodontic sequence of $D$.

\begin{example}
The orthodontic sequence for the Rothe diagram $D(31542)$ is shown in Figure~\ref{fig:31542-double-orth}.
\begin{figure}[ht]
\centering
\includegraphics[scale=0.8]{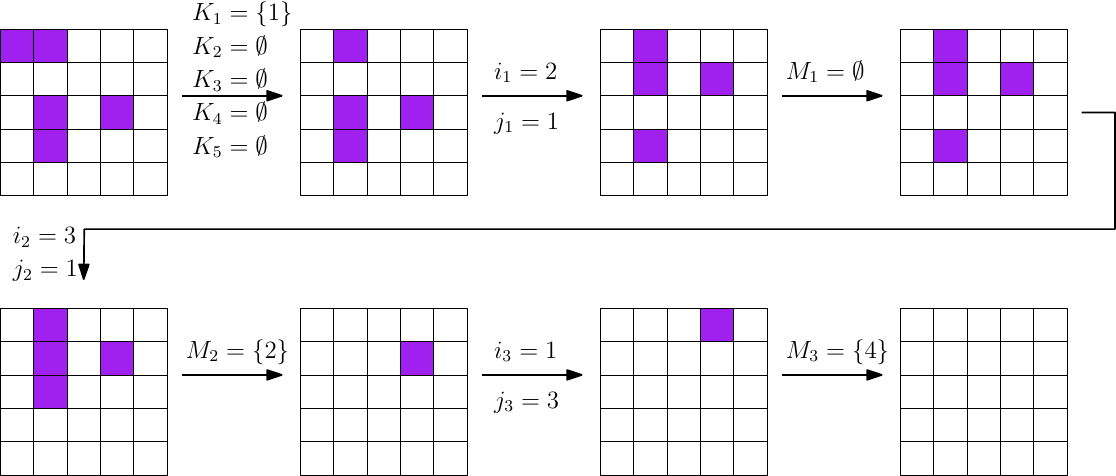}
\caption{The double orthodontic sequence for $w = 31542$.}
\label{fig:31542-double-orth}
\end{figure}
\end{example}

For a polynomial $f\in \CC[x_1, \dots, x_n, y_1, \dots, y_n]$, define the operators
\begin{equation}
\label{eqn:pi-ij-operators}
\pi_{i,j}(f) \colonequals \del_i((x_i + y_j)f), \qquad \overline\pi_{i,j}(f) \colonequals \overline\del_i((x_i + y_j - x_iy_j)f).
\end{equation}
For $i\in[n]$ and $M\subseteq[n]$, define the quantities
\begin{equation}
\label{eqn:omega-iM-operators}
\omega_i^M\colonequals \prod_{\substack{j\in[i]\\m\in M}}(x_j + y_m),\qquad\overline\omega_i^M\colonequals \prod_{\substack{j\in [i]\\m\in M}} (x_j + y_m - x_jy_m).
\end{equation}
\begin{defn}Let $D$ be a \%-avoiding diagram with double orthodontic sequence $\mathbf K, \mathbf i, \mathbf j, \mathbf M$. Define
\[
\mathscr G_D(\mathbf x;\mathbf y)\colonequals \overline\omega_1^{K_1}\overline\omega_2^{K_2}\dots\overline\omega_n^{K_n}\overline\pi_{i_1,j_1}(\overline\omega_{i_1}^{M_1}\overline\pi_{i_2,j_2}(\overline\omega_{i_2}^{M_2}\dots\overline\pi_{i_\ell,j_\ell}(\overline\omega_{i_\ell}^{M_\ell})\dots)),
\]
and write
\[
\mathscr S_D(\mathbf x;\mathbf y)\colonequals \omega_1^{K_1}\omega_2^{K_2}\dots\omega_n^{K_n}\pi_{i_1,j_1}(\omega_{i_1}^{M_1}\pi_{i_2,j_2}(\omega_{i_2}^{M_2}\dots\pi_{i_\ell,j_\ell}(\omega_{i_\ell}^{M_\ell})\dots)).\qedhere
\]
\end{defn}
\begin{rem}
\label{rem:sd-nonempty}
It can be shown by inducting down the orthodontic sequence that the monomial $\prod_{(i,j)\in D}x_i$ appears in $\mathscr S_D$ with coefficient $1$. As $\mathscr S_D$ is nonzero, it follows that $\mathscr S_D$ is the lowest degree part of $\mathscr G_D$.
\end{rem}
\begin{rem}
The polynomial $\mathscr G_D(\mathbf x; \mathbf y)$ is not invariant under permuting columns, in contrast to previous work \cite{magyar98, mss22} on the orthodontia algorithm. This is by design: the Rothe diagram $D(2413)$ can be obtained from the Rothe diagram $D(132)$ by permuting the columns and adding a column equal to $\{1,2\}$, and although the ordinary Grothendieck polynomials satisfy $\mathfrak G_{2413}(\mathbf x) = x_1x_2 \mathfrak G_{132}(\mathbf x)$, there does not exist a polynomial $g(\mathbf x; \mathbf y)$ so that $\mathfrak G_{2413}(\mathbf x; \mathbf y) = g(\mathbf x; \mathbf y) \cdot \mathfrak G_{132}(\mathbf x; \mathbf y)$.
\end{rem}

\begin{rem}
In our construction, $\mathbf j$ may involve negative indices, and so the polynomial $\mathscr G_D(\mathbf x;\mathbf y)$ (and $\mathscr S_D(\mathbf x;\mathbf y)$) may involve variables $y_j$ for $j \leq 0$. Part of our claim in Theorem~\ref{thm:double-groth-master} is that no such $y_j$ for $j \leq 0$ appear when $D = D(w)$ is a Rothe diagram. In Theorem~\ref{thm:lascoux-positivity-main} and Conjecture~\ref{conj:lascoux-positivity-general}, all $y$-variables (even those witth negative indices) are set to $-1$.
\end{rem}

\subsection*{Orthodontic sort order on permutations}
Following \cite{mss22}, we recall some basic facts about sorted permutations and the sorting projection.

In what follows, a \emph{standard interval} is a set of the form $[j]$ for some $j\geq 0$. Recall that a permutation $w \in S_n$ is called \emph{dominant} if it is $132$-avoiding. A permutation $w$ is dominant if and only if its Rothe diagram $D(w) = (D(w)_1, \dots, D(w)_n)$ consists only of standard intervals, which necessarily satisfy $\#D(w)_1 \geq \dots \geq \#D(w)_n$.
\begin{defn}
Fix a permutation $w \in S_n$. The \emph{primary column data} of $w$, denoted $(h,C,\alpha,i_1,\beta)$, is defined as follows.

If $w$ is not dominant, the diagram $D(w)$ has a column which is not a standard interval. Set $h$ to be the smallest integer such that the column $D(w)_{h+1}$ is not a standard interval. Set $C$ to be the column $D(w)_{h+1}\subseteq [n]$. Set $\alpha$ to be the largest integer such that $[\alpha]\subseteq C$. Set $i_1$ to be the smallest missing tooth of $C$. Lastly, set $\beta=i_1 - \alpha$, the size of the ``uppermost gap'' of $C$.

If $w$ is dominant, set $h=n$, $C=\emptyset$, $\alpha=0$, $i_1=n$, and $\beta=n$. 
\end{defn}
\begin{lem}[{\cite[Lem 3.4]{mss22}}]
\label{lem:sigma-sort}
Any permutation $w$ restricts to a bijection $[\alpha+1, i_1]\to [h-\beta+1,h]$. The corresponding permutation $\sigma\in S_\beta$ is dominant.
\end{lem}
See Figure~\ref{fig:D68342751-and-sort} for an example.

For $w\in S_n$, write $\sigma(w)$ for the dominant permutation obtained by restricting $w$ to $[\alpha+1, i_1]$.
\begin{defn}
The permutation $w$ is \emph{sorted} if $\sigma(w)$ is the identity. The \emph{sorting} of $w$, denoted $w_\sort$, is the permutation obtained from $w$ by reordering $w(\alpha+1), \dots, w(i_1)$ to be in increasing order.
\end{defn}

\begin{example}
\label{ex:running-1}
Let $w = 68342751$. The primary column data of $w$ is 
\[h=4,\,\,C=\{1,2,6\},\,\,\alpha=2,\,\,i_1=5,\,\,\beta=3. \] 
Note that $w$ restricts to a bijection $ \{3,4,5\}\to\{2,3,4\}$; the corresponding permutation $\sigma(w) \in S_3$ is $\sigma(w) = 231$. The sorting of $w$ is $w_\sort = 68234751$. The Rothe diagrams of $w$ and $w_\sort$ are displayed in Figure~\ref{fig:D68342751-and-sort}.
\begin{figure}[ht]
\centering
\includegraphics{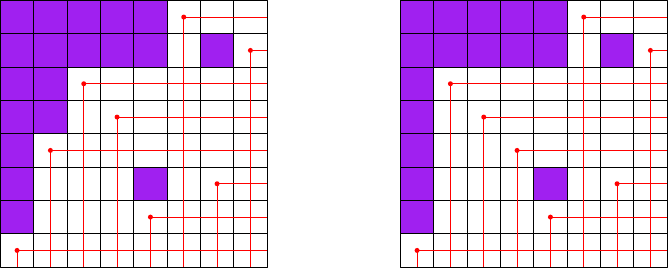}
\caption{The Rothe diagram of $w = 68342751$ on the left, and of its sorting $w_\sort = 68234751$ on the right. In this example, $\sigma(w) = 231$.}
\label{fig:D68342751-and-sort}
\end{figure}
\end{example}

\begin{defn}
\label{defn:os}
The \emph{orthodontic sort order} $\leq_\os$ is the reflexive and transitive closure of the relations
\begin{align*}
\label{eqn:os1}
w_{\textup{sort}}&\preceq w \tag{$\dagger$}\\
\label{eqn:os2} ws_{i_1}\dots s_{\alpha+1}&\preceq w\textup{ whenever $w$ is nonidentity and sorted}. \tag{$\ddagger$}
\end{align*}
\end{defn}
\begin{prop}[{\cite[Prop 5.6]{mss22}}]
The relation $\leq_\os$ is a partial order on $S_n$ and the identity permutation is the minimum element.
\end{prop}
\section{Orthodontia and double Grothendieck polynomials}
We establish Theorem~\ref{thm:double-groth-master} by studying the relationship between $\mathfrak G_w$ and $\mathfrak G_{w_\sort}$ (Proposition~\ref{prop:os1-groth}), the relationship between $\mathscr G_{D(w)}$ and $\mathscr G_{D(ws_{i_1}\dots s_{\alpha+1})}$ (Theorem~\ref{thm:os2-orth}), and then inducting on the orthodontic sort order (Definition~\ref{defn:os}). 

The following example motivates Proposition~\ref{prop:rothe-of-sort} and Corollary~\ref{cor:os1-orth}.
\begin{example}
\label{ex:running-2}
Let $w = 68342751$, as in Example~\ref{ex:running-1}, so that $w_\sort = 68234751$. Observe that $D(w_\sort) = D(w) \setminus \{(2,3), (2,4)\}$ is obtained from $D(w)$ by removing bottom-aligned portions of standard interval columns. Proposition~\ref{prop:rothe-of-sort} asserts that this phenomenon holds in general.

Write $\mathbf K(w) = (K_1, \dots, K_n)$ and $\mathbf K(w_\sort) = (K_1', \dots, K_n')$. Then:
\[
K_1 = \emptyset, \quad K_2 = \{3,4\}, \quad K_3 = \emptyset, \quad K_4 = \{2\}, \quad K_5 = \emptyset, \quad K_6 = \emptyset, \quad K_7 = \{1\}
\]
and
\[
K_1' = \emptyset, \quad K_2' = \{2,3,4\}, \quad K_3' = \emptyset, \quad K_4' = \emptyset, \quad K_5' = \emptyset, \quad K_6' = \emptyset, \quad K_7' = \{1\}.
\]
The fact that $K_4' = K_4 \setminus \{2\}$ and $K_2' = K_2 \cup \{2\}$ can be viewed as a consequence of the fact that $D(\sigma(w)) = D(231)$ has Rothe diagram consisting of one standard interval column of size two.

As $D(w_\sort)$ can be obtained from $D(w)$ by removing bottom-aligned portions of standard interval columns, the two diagrams agree after removing standard interval columns. It follows that the $\mathbf i,\mathbf j, \mathbf M$ orthodontic sequences of $w$ and $w_\sort$ agree: $\mathbf i(w) = \mathbf i(w_\sort)$, $\mathbf j(w) = \mathbf j(w_\sort)$, and $\mathbf M(w) = \mathbf M(w_\sort)$. This phenomenon holds more generally, and is stated precisely in Corollary~\ref{cor:os1-orth}.
\end{example}

\begin{prop}
\label{prop:rothe-of-sort}
Let $(h,C,\alpha,i_1,\beta)$ denote the primary column data of $w\in S_n$. Set $\lambda^\sigma_j \colonequals \#D(\sigma(w))_j$. The Rothe diagram $D(w_\sort)$ can be obtained from $D(w)$ by removing bottom-aligned portions of standard-interval columns, and
\begin{equation}
\label{eqn:rothe-of-sort}
D(w) \setminus D(w_\sort) = \{(\alpha+a,h-\beta+b)\colon 1 \leq a\leq\lambda^\sigma_b \textup{ and } 1\leq b\leq \beta\}.
\end{equation}
\end{prop}

\begin{proof}[Proof of Proposition~\ref{prop:rothe-of-sort}]
As $w_\sort$ is obtained from $w$ by reordering numbers $\{w^{-1}(j)\colon j \in [h-\beta+1,h]\}$, the columns $D(w_\sort)_j$ and $D(w)_j$ agree whenever $j\not\in [h-\beta+1,h]$. By definition of the primary column data and of the sorting of a permutation, $D(w)_{h-\beta+b} = [\alpha+\lambda^\sigma_b]$ and $D(w_\sort)_{h-\beta+b} = [\alpha]$.
\end{proof}
\begin{cor}[cf.\ {\cite[Prop 4.4]{mss22}}]
\label{cor:os1-orth}
Let $(h,C,\alpha,i_1,\beta)$ denote the primary column data of $w\in S_n$. Set 
\[
\mathcal S_a \colonequals \left\{h-\beta+j\colon\begin{array}{c} j\in [\beta]\\ \#D(\sigma(w))_j = a\end{array}\right\}.
\]
Write $\mathbf K(w) = (K_1, \dots, K_n)$ and $\mathbf K(w_\sort) = (K_1', \dots, K_n')$. Then
\[
\mathbf i(w) = \mathbf i(w_\sort), \qquad \mathbf j(w) = \mathbf j(w_\sort), \qquad \mathbf M(w) = \mathbf M(w_\sort),
\]
and
\[
K_j = \begin{cases} K_j' &\textup{ if } j\leq \alpha-1,\\ K_j' \setminus \{h-\beta+1, \dots, h \} \cup \mathcal S_0 &\textup{ if } j = \alpha,\\ K_j' \cup \mathcal S_a &\textup{ if } j = \alpha + a, \textup{ where } a\in[\beta],\\
K_j' &\textup{ if } j \geq i_1 + 1. \end{cases}
\]
In particular,
\[
\mathscr G_{D(w)} = \left(\prod_{(a,b)\in D(w)\setminus D(w_\sort)} (x_a + y_b - x_ay_b)\right) \cdot\mathscr G_{D(w_\sort)}.
\]
\end{cor}
\begin{proof}
Proposition~\ref{prop:rothe-of-sort} implies that the diagrams $D(w)$ and $D(w_\sort)$ agree after removing their standard interval columns. It follows that the orthodontic sequences $\mathbf i, \mathbf j, \mathbf M$ of $w$ and $w_\sort$ agree. The formulas for $K_j$ in terms of $K_j'$ hold by~\eqref{eqn:rothe-of-sort}. The relationship between $\mathscr G_{D(w)}$ and $\mathscr G_{D(w_\sort)}$ follows from the construction~\eqref{eqn:double-groth-master} of $\mathscr G_D$.
\end{proof}
\begin{lem}
\label{lem:CwEw}
Let $(h,C,\alpha,i_1,\beta)$ denote the primary column data of $w\in S_n$. For $j\in[h]$, set 
\[\lambda_j\colonequals \#D(w)_j \quad\mbox{and}\quad \nu_j\colonequals \#\{i\leq j\colon \#D(w)_i \geq w^{-1}(j)\}.\]
Fix any $P \in \PD(w)$. Then for every $j\in[h]$, the pipe $P^{(j)}$ begins by traversing $\lambda_j$ crossing tiles vertically, ends by traversing $\nu_j$ cross tiles horizontally, and otherwise traverses only bump tiles.
In particular, every tile in
\[
C_w \colonequals \left\{(i,j)\colon \begin{array}{c}1 \leq i \leq \lambda_j,\\ 1\leq j \leq h\end{array}\right\}
\]
is a cross, and every tile in
\[
E_w \colonequals \left\{(i,j)\colon \begin{array}{c} h-\beta+1 \leq i \leq h,\\\alpha+1 \leq j \leq i_1,\\i + j \leq h + \alpha + 1\end{array}\right\} \setminus C_w
\]
is an elbow.
\end{lem}
\begin{example}
Figure~\ref{fig:typical-pd} gives an example of a typical pipe dream, with the structure guaranteed by Lemma~\ref{lem:CwEw} highlighted.
\begin{figure}[ht]
\centering
\includegraphics{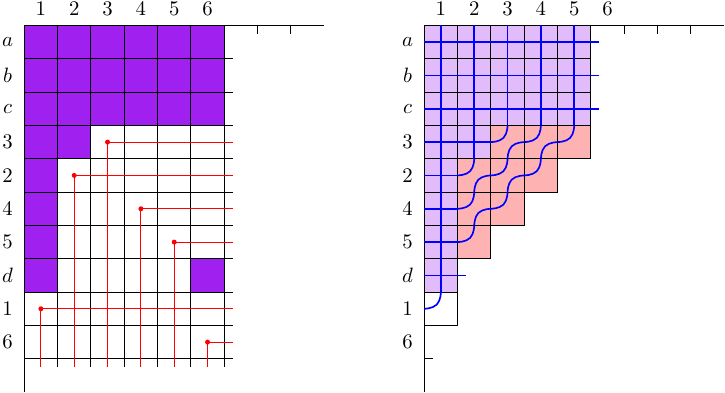}
\caption{Left: The Rothe diagram of a permutation $w = abc3245d16\ldots$, for $a,b,c,d > 6$. Right: A typical pipe dream $P\in\PD(w)$. The boxes in $C_w$ are shaded purple and the boxes in $E_w$ are shaded red.}
\label{fig:typical-pd}
\end{figure}
\end{example}
\begin{proof}[Proof of Lemma~\ref{lem:CwEw}]
Induct on $j$. The pipe $P^{(1)}$ necessarily traverses $\lambda_1$ cross tiles vertically before traversing a bump tile to exit at the $w^{-1}(1)$\textsuperscript{th} row. Now assume that, for all $k < j$, the pipes $P^{(k)}$ traverse $\lambda_k$ cross tiles vertically, then traverse bump tiles, before ending by traversing $\nu_k$ tiles horizontally.

As $\lambda_1 \geq \dots \geq \lambda_j$, the inductive hypothesis guarantees that all tiles
\[
\left\{(i,k)\colon \begin{array}{c} 1 \leq i \leq \lambda_j \\ 1 \leq k \leq j-1\end{array}\right\}
\]
are cross tiles. Thus if $P^{(j)}$ traverses a bump tile after traversing $a < \lambda_j$ cross tiles vertically, pipe $P^{(j)}$ is forced to travel horizontally until exiting at the $(a+1)$\textsuperscript{th} row. As $a + 1 \leq \lambda_j < w^{-1}(j)$, this is a contradiction.

As $\lambda_i$ is weakly decreasing, $\#D(w)_i \geq w^{-1}(j)$ if and only if $i \in [\nu_j]$. The inductive hypothesis guarantees that the pipe exiting at row $w^{-1}(j)$ traverses $\nu_j$ crossing tiles horizontally before traversing any bump tiles.

It remains to show that $P^{(j)}$ does not traverse any other cross tiles. To this end, observe that any such cross tile $(a,b)$ satisfies $a > \lambda_j$ and $b > \nu_j$. It follows that any such cross tile cannot involve $P^{(w(k))}$ for $k \leq \lambda_j$ as the paths of these pipes are contained in $\{(a,b)\colon a \leq k\}$, and cannot involve $P^{(k)}$ for $k \leq \nu_j$ as the paths of these pipes are contained in $\{(a,b)\colon b \leq \nu_j\}$. In other words, any such cross tile involves $P^{(j)}$ and $P^{(k)}$ for $k\not\in w([\lambda_j])$ and $k\not\in[\nu_j]$.

Then, the first part of the result follows from the fact 
%that the inversions of $w$ involving $j$ are
%\[
%\{(k,j)\colon k \in w([\lambda_j]) \textup{ or } k \in[\nu_j]\}.
%\]
the labels of the pipes crossing $P^{(j)}$ are 
\[w([\lambda_j])\cup [\nu_j].\]
The rest of the lemma follows from the fact that the set of pipes entering the top edge of the triangular grid
\[
\left\{(i,j)\colon \begin{array}{c} h-\beta+1 \leq i \leq h,\\\alpha+1 \leq j \leq i_1,\\i + j \leq h + \alpha + 1\end{array}\right\},
\]
namely the set $\{P^{(j)}\colon h-\beta+1\leq j\leq h\}$, is equal to the set of pipes exiting the left edge, hence restricts to a pipe dream for $\sigma(w)$.
\end{proof}

\begin{cor}
\label{cor:crossings}
Let $(h,C,\alpha,i_1,\beta)$ denote the primary column data of $w\in S_n$, and let $\mathcal S\colonequals \{h-\beta+1, \dots, h\}$. Fix $P \in \PD(w)$.
\begin{itemize}
\item For $i,j\in\mathcal S$ and $k\not\in\mathcal S$, the pipes $P^{(i)}$ and $P^{(k)}$ cross if and only if $P^{(j)}$ and $P^{(k)}$ cross.
\item For $i\in\mathcal S$ and $j\in[n]$, the pipes $P^{(i)}$ and $P^{(j)}$ cross at most once.
\item For $i,j\in\mathcal S$, the pipes $P^{(i)}$ and $P^{(j)}$ can cross only at $C_w \setminus C_{w_\sort}$, and conversely every tile in $C_w \setminus C_{w_\sort}$ is a crossing of pipes $P^{(i)}$ and $P^{(j)}$ for $i,j\in\mathcal S$.
\end{itemize}
\end{cor}
\begin{proof}
For any $i\in\mathcal S$, Lemma~\ref{lem:CwEw} guarantees that pipe $P^{(i)}$ begins by traversing $\alpha$ many cross tiles vertically, crossing pipes $P^{(k)}$ for $k \in w([\alpha])$, and ends by traversing $h-\beta$ many cross tiles horizontally, crossing pipes $P^{(k)}$ for $k\in [h-\beta]$. Lemma~\ref{lem:CwEw} also guarantees that the tiles in $(C_w \cup E_w) \setminus C_{w_\sort}$ involve only the pipes $P^{(i)}$ for $i\in\mathcal S$. All three parts of the claim follow.
\end{proof}

\begin{prop}[cf.\ {\cite[Prop 3.9]{mss22}}]
\label{prop:os1-groth}
Let $w\in S_n$. Then
\[
\mathfrak G_w = \left(\prod_{(a,b) \in D(w)\setminus D(w_\sort)} (x_a + y_b - x_ay_b) \right)\cdot\mathfrak G_{w_{\sort}}.
\]
\end{prop}
\begin{proof}
Fix a pipe dream $P \in \PD(w)$. By Proposition~\ref{prop:rothe-of-sort} and Lemma~\ref{lem:CwEw}, the tiles in $D(w) \setminus D(w_\sort)$ are crosses. Let $F(P)$ be the pipe dream obtained from $P$ by replacing the crosses in $D(w)\setminus D(w_\sort)$ with elbows. By uncrossing all pipes $P^{(i)}$ and $P^{(j)}$ with $i,j\in\mathcal S$, Corollary~\ref{cor:crossings} guarantees that the Demazure word $\del(F(P))$ of $F(P)$ is $w_\sort$.

For any pipe dream $P\in\PD(w_\sort)$, let $G(P)$ be the pipe dream obtained from $P$ by replacing the elbows in $D(w)\setminus D(w_\sort)\subseteq E_{w_\sort}$ by crosses. Corollary~\ref{cor:crossings} guarantees that the Demazure word $\del(G(P))$ is $w$. As $FG \colon \PD(w_\sort)\to\PD(w_\sort)$ and $GF\colon \PD(w)\to \PD(w)$ are the identity, the map $F$ is a bijection.

The claim follows from the fact that $F\colon \PD(w)\to\PD(w_\sort)$ scales the weight of a pipe dream by the monomial
\[
\prod_{(a,b) \in D(w)\setminus D(w_\sort)} (x_a + y_b - x_ay_b).\qedhere
\]
\end{proof}

The following example motivates Theorem~\ref{thm:os2-orth}, which gives a precise description of the first few terms of the orthodontic sequences of sorted permutations.
\begin{example}[cf.\ {\cite[Ex 4.1]{mss22}}]
Consider the sorted permutation $w = 68234751$, with Rothe diagram depicted on the left of Figure~\ref{fig:68234751-seq}. The Rothe diagram of $w' \colonequals ws_5s_4s_3$ ``almost'' appears in the double orthodontic sequence for $D(w)$, as shown in the top of Figure~\ref{fig:68234751-seq}. Part (6) of Theorem~\ref{thm:os2-orth} makes this connection precise.

\begin{figure}
\begin{center}
\hspace*{-5ex}\includegraphics[scale=1.03]{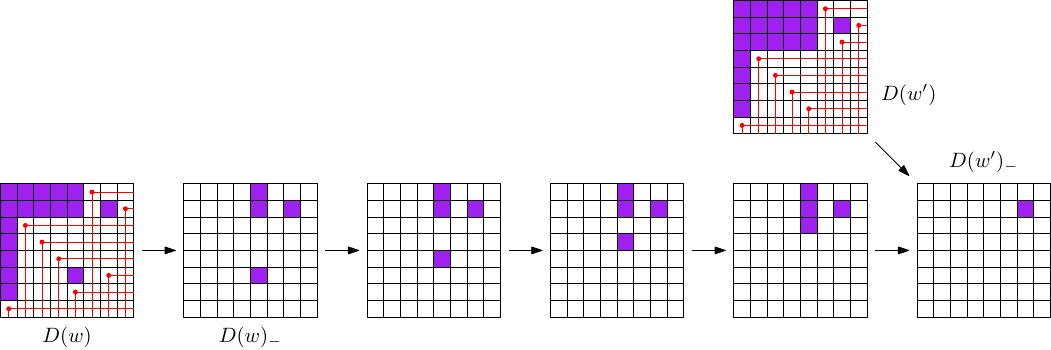}
\end{center}
\caption{Left: Rothe diagram $D(w)$ for $w = 68234751$. Middle: first few orthodontic steps for $D(w)$. Top: The Rothe diagram $D(w')$ for $w' = ws_5s_4s_3 = 68723451$.}
\label{fig:68234751-seq}
\end{figure}

The primary column data $(h,C,\alpha,i_1,\beta)$ of $w$ is given by $h = 4$, $C = \{1,2,6\}$, $\alpha = 2$, $i_1 = 5$, and $\beta = 3$.
From the fact that the uppermost gap has size $\beta = 3$ and the first missing tooth $i_1 = 5$ occurs in column $h + 1 = 5$, it follows that the next missing teeth are $i_2 = i_1 - 1 = 4$ and $i_3 = i_2 - 1 = 3$, and both occur in the fifth column. This phenomenon is captured as part (1) of Theorem~\ref{thm:os2-orth}. Direct computation gives $j_1 = 5 - 3 = 2$, $j_2 = j_1 + 1 = 3$, and $j_3 = j_2 + 1 = 4$; this is part (2) of Theorem~\ref{thm:os2-orth}. None of these row-swap orthodontic moves introduce standard interval columns; this is part (5) of Theorem~\ref{thm:os2-orth}.

As $w$ is sorted, the three columns immediately to the left of column $h+1 = 5$ (i.e., columns $2,3,4$) are standard intervals equal to $[\alpha] = [2]$. This is part (3) of Theorem~\ref{thm:os2-orth}. Furthermore, there are no standard interval columns of size $k$ for any $k \in [\alpha+1, i_1] = \{3,4,5\}$. This is part (4) of Theorem~\ref{thm:os2-orth}.
\end{example}

\begin{thm}[cf.\ {\cite[Thm 4.2]{mss22}}]
\label{thm:os2-orth}
Let $w\in S_n$ be a nonidentity sorted permutation, and suppose $w$ has orthodontic sequence 
\[
\mathbf i = (i_1, \dots, i_\ell), \qquad \mathbf j = (j_1, \dots, j_\ell), \qquad \mathbf K = (K_1, \dots, K_n),\qquad \mathbf M = (M_1, \dots, M_\ell).
\]
Let $(h,C,\alpha,i_1,\beta)$ be the primary column data of $w$. Write $\mathcal S\colonequals \{h-\beta+1,\dots,h\}$. Then:
\begin{enumerate}
\item For $k\in[\beta]$, we have $i_k = i_1 - k + 1$.
\item For $k\in[\beta]$, we have $j_k = h - \beta + k$.
\item If $\alpha > 0$, then $K_\alpha \supseteq \mathcal S$.
\item For $k\in[\alpha+1,i_1]$, we have $K_k = \emptyset$.
\item For $k\in[\beta-1]$, we have $M_k = \emptyset$. 
\item The permutation $w'\colonequals ws_{i_1}\dots s_{\alpha+1}$ has orthodontic sequence
\[
\mathbf i(w') = (i_{\beta+1}, \dots, i_\ell), \qquad \mathbf j(w') = (j_{\beta+1}, \dots, j_\ell), \qquad \mathbf M(w') = (M_{\beta+1},\dots,M_\ell),
\]
and
\[
\mathbf K(w') = \begin{cases} (K_1, \dots, K_{\alpha-1}, K_\alpha\setminus \mathcal S,\mathcal S\sqcup M_\beta, K_{\alpha+2}, \dots, K_n) &\textup{ if } \alpha > 0\\ (\mathcal S \sqcup M_\beta, K_2, \dots, K_n) &\textup{ if } \alpha = 0.\end{cases}
\]
In particular,
\begin{equation}
\label{eqn:Gw-vs-Gw'}
\mathscr G_{D(w)} = \overline\pi_{i_1,h+1-\beta} \dots \overline\pi_{\alpha+1,h}\left(\prod_{s\in \mathcal S} (x_{\alpha + 1} + y_s - x_{\alpha+1}y_s)^{-1} \cdot \mathscr G_{D(w')}\right).\tag{$\star$}
\end{equation}
\end{enumerate}
\end{thm}
\begin{proof}
By definition, $C = D(w)_{h+1}$ contains $[\alpha]\cup\{i_1 + 1\}$ and does not contain any of $\alpha + 1, \dots, i_1$. Thus, the leftmost missing teeth of the diagrams in the orthodontic sequence of $D(w)$ are equal to $(i_1, i_1 - 1, \dots, \alpha+1)$, and all occur at column $h+1$, with corresponding sequence of uppermost gaps $(\beta, \beta - 1, \dots, 1)$. Parts (1) and (2) follow. 

As $w$ is sorted, the columns $D(w_\sort)_{h-\beta+b}$ are equal to $[\alpha]$, and part (3) follows. Parts (4) and (5) are proven as \cite[Thm 4.2, (iii), (iv)]{mss22} respectively. 

In order to relate the orthodontic sequences of $w$ and $w'$, consider the diagrams
\begin{align*}
E &= (\underbrace{\emptyset, \dots, \emptyset}_{\textup{$h$ many}}, D(w)_{h+1}, \dots, D(w)_n),\\
E' &= (\underbrace{\emptyset, \dots, \emptyset}_{\textup{$h$ many}}, D(w')_{h+1}, \dots, D(w')_n).
\end{align*}
Note that $E' = E\cdot s_{i_1} \dots s_{\alpha+1}$ and that $M_\beta = \{j\colon E'_j = [\alpha+1]\}$. As the first $h$ columns of $D(w)$ and $D(w')$ are standard intervals, it follows that $D(w')_-$ occurs in the orthodontic sequence of diagrams for $D(w)$ and furthermore that
\[
\mathbf i(w') = (i_{\beta+1}, \dots, i_\ell), \qquad \mathbf j(w') = (j_{\beta+1}, \dots, j_\ell), \qquad \mathbf M(w') = (M_{\beta+1},\dots,M_\ell).
\]
From the facts that 
\begin{itemize}
	\item $D(w)_j = D(w')_j$ for $j \leq h-\beta$,
	\item $D(w)_j \cup \{\alpha+1\} = D(w')_j$ for $h-\beta+1\leq j\leq h$, and
	\item $M_\beta = \{j\colon E'_j = [\alpha+1]\}$,
\end{itemize}
the formula for $\mathbf K(w')$ in terms of $\mathbf K(w)$ follows. This gives part (6).

Finally, the equation~\eqref{eqn:Gw-vs-Gw'} is a consequence of part (4) of the claim via the string of equalities
\begin{align*}
\mathscr G_{D(w)} &= \overline\omega_1^{K_1}\dots\overline\omega_n^{K_n}\overline\pi_{i_1,h+1-\beta}\dots\overline\pi_{\alpha+1,h}(\mathscr G_{D(w')_-})
\\&\overset{(4)}=\overline\pi_{i_1,h+1-\beta}\dots\overline\pi_{\alpha+1,h}\left(\overline\omega_1^{K_1}\dots\overline\omega_n^{K_n}\mathscr G_{D(w')_-}\right)
\\&= \overline\pi_{i_1,h+1-\beta} \dots \overline\pi_{\alpha+1,h}\left(\prod_{s\in \mathcal S} (x_{\alpha + 1} + y_s - x_{\alpha+1}y_s)^{-1} \cdot \mathscr G_{D(w')}\right).\qedhere
\end{align*}
\end{proof}

In order to relate the formula from Theorem~\ref{thm:os2-orth} to Grothendieck polynomials, the following lemmas will be useful.

\begin{lem}
\label{lem:operator-trick}
Let $\overline\Delta_a\colonequals\overline\del_{i+a}\circ\dots\circ \overline\del_i$. For $a = -1$, set $\overline\Delta_a \colonequals 1$ to be the identity operator. Given $g\in\CC[\mathbf x,\mathbf y]$ and $\ell \geq 1$ such that $\del_{i+\ell}(g) = 0$, the equality
\[
\overline\pi_{i+\ell, j_\ell}\circ \overline\Delta_{\ell-1}(g) = \overline\del_{i+\ell}\circ \overline\pi_{i+\ell-1, j_\ell}\circ \overline\Delta_{\ell-2}(g)
\]
holds. In particular, for any $\ell \geq 0$,
\[
\overline\pi_{i+\ell, j_\ell}\circ \overline\Delta_{\ell-1}(g) = \overline\Delta_\ell((x_i + y_{j_\ell} - x_i y_{j_\ell})g)
\]
whenever $\del_{i+1}(g) = \dots = \del_{i+\ell}(g) = 0$.
\end{lem}
\begin{proof}
Work in the ring of endomorphisms of $\CC[\mathbf x, \mathbf y]$. For $h\in\CC[\mathbf x, \mathbf y]$, write $m_h$ for the ``multiplication by $h$'' operator, given by $f\mapsto fh$. As
\[
\del_i \circ m_{x_i} = (m_{x_{i+1}}\circ \del_i) + 1 \qquad\textup{ and } \qquad \del_i\circ m_{y_j} = m_{y_j}\circ \del_i\textup{ for all $j$},
\]
observe that
\begin{align*}
m_{x_{i+1} + y_j - x_{i+1}y_j}\circ \del_i &= (m_{x_{i+1}}\circ m_{1-y_j} + m_{y_j}) \circ \del_i \\&= (\del_i\circ m_{x_i} - 1)\circ m_{1-y_j} + \del_i\circ m_{y_j}\\&\label{eqn:commute}= \del_i\circ (m_{x_i + y_j - x_iy_j}) - m_{1-y_j}.\tag{$\diamondsuit$}
\end{align*}
Hence
\begin{align*}
\overline\pi_{i+\ell, j_\ell}\circ \overline\Delta_{\ell-1} &= \overline\del_{i+\ell}\circ m_{x_{i+\ell} + y_{j_\ell} - x_{i+\ell}y_{j_\ell}}\circ \del_{i+\ell-1}\circ m_{1 - x_{i+\ell}} \circ \overline\Delta_{\ell-2}\\
&\overset{\eqref{eqn:commute}}=\overline\del_{i+\ell} \circ (\del_{i+\ell-1} \circ m_{x_{i+\ell-1} + y_{j_\ell} - x_{i+\ell-1}y_{j_\ell}} - m_{1-y_{j_\ell}}) \circ m_{1-x_{i+\ell}}\circ\overline\Delta_{\ell-2}\\
\label{eqn:operator-trick}
&=\overline\del_{i+\ell}\circ\overline\pi_{i+\ell-1,j_\ell}\circ\overline\Delta_{\ell-2} - \overline\del_{i+\ell}\circ m_{1-y_{j_\ell}}\circ m_{1-x_{i+\ell}} \circ\overline\Delta_{\ell-2}.\tag{$*$}
\end{align*}
Note that $\overline\del_{i+\ell}$ and $m_{1-x_{i+\ell}}$ commute with $\overline\Delta_{\ell-2}$, so that
\begin{align*}
\overline\del_{i+\ell}\circ m_{1-y_{j_\ell}}\circ m_{1-x_{i+\ell}} \circ\overline\Delta_{\ell-2}(g) &= m_{1-y_{j_\ell}}\circ \overline\Delta_{\ell-2}\circ\overline\del_{i+\ell}\circ m_{1-x_{i+\ell}} (g)
\\&=\label{eqn:vanishing-trick}m_{1-y_{j_\ell}}\circ \overline\Delta_{\ell-2}\circ\underbrace{\del_{i+\ell}((1-x_{i+\ell+1})(1-x_{i+\ell})g)}_{=0}.\tag{$**$}
\end{align*}
Applying the equality~\eqref{eqn:operator-trick} of operators to the function $g$, and using~\eqref{eqn:vanishing-trick}, the claimed equality
\[
\overline\pi_{i+\ell, j_\ell}\circ \overline\Delta_{\ell-1}(g) = \overline\del_{i+\ell}\circ \overline\pi_{i+\ell-1, j_\ell}\circ \overline\Delta_{\ell-2}(g)
\]
follows.

The second part of the lemma, for $\ell = 0$, is nothing but the identity
\begin{equation}
\label{eqn:ell0-case}
\overline\pi_{i,j_0}(g) = \overline\del_i((x_i + y_{j_0} - x_iy_{j_0})g),
\end{equation}
and in general it follows from repeated application of the first part of the lemma combined with the identity~\eqref{eqn:ell0-case}.
\end{proof}

\begin{lem}[cf.\ {\cite[Lem 5.11]{mss22}}]
\label{lem:pi-to-del}
Let $g$ be a polynomial with
\[
\del_{i+1}(g) = \dots = \del_{i+k}(g) = 0.
\]
Then
\[
\overline\pi_{i+k,j_k} \overline\pi_{i+k-1,j_{k-1}} \dots \overline\pi_{i,j_0}(g) = \overline\del_{i+k}\dots\overline\del_i\left(\prod_{a=0}^k(x_i + y_{j_a} - x_iy_{j_a}) \cdot g\right).
\]
\end{lem}
\begin{proof}
Induct on $k$ as follows. The base case $k = 0$ is the identity~\eqref{eqn:ell0-case}. For $k > 0$, the inductive hypothesis guarantees that
\[
\overline\pi_{i+k-1,j_{k-1}} \dots \overline\pi_{i,j_0}(g) = \overline\Delta_{k-1}\left(\prod_{a=0}^{k-1}(x_i + y_{j_a} - x_iy_{j_a}) \cdot g\right).
\]
Then, because
\[
\del_{i+1}\left(\prod_{a=0}^{k-1}(x_i + y_{j_a} - x_iy_{j_a})g\right) = \dots = \del_{i+k}\left(\prod_{a=0}^{k-1}(x_i + y_{j_a} - x_iy_{j_a})g\right) = 0,
\]
Lemma~\ref{lem:operator-trick} implies that
\[
\overline\pi_{i+k,j_k} \circ \overline\Delta_{k-1}\left(\prod_{a=0}^{k-1}(x_i + y_{j_a} - x_iy_{j_a})\cdot g\right) =\overline\Delta_k\left(\prod_{a=0}^k(x_i + y_{j_a} - x_iy_{j_a})\cdot g\right).\qedhere
\]
\end{proof}

\newtheorem*{thm:double-groth-master}{Theorem~\ref{thm:double-groth-master}}
\begin{thm:double-groth-master}
Let $D$ be a \%-avoiding diagram with double orthodontic sequence $\mathbf K, \mathbf i, \mathbf j, \mathbf M$. Define
\[
\mathscr G_D(\mathbf x;\mathbf y)\colonequals \overline\omega_1^{K_1}\overline\omega_2^{K_2}\dots\overline\omega_n^{K_n}\overline\pi_{i_1,j_1}(\overline\omega_{i_1}^{M_1}\overline\pi_{i_2,j_2}(\overline\omega_{i_2}^{M_2}\dots\overline\pi_{i_\ell,j_\ell}(\overline\omega_{i_\ell}^{M_\ell})\dots)).\tag{\ref{eqn:double-groth-master}}
\]
When $D = D(w)$ is the Rothe diagram of a permutation, then $\mathscr G_D(\mathbf x; \mathbf y) = \mathfrak G_w(\mathbf x; \mathbf y)$.
\end{thm:double-groth-master}
\begin{proof}[Proof of Theorem~\ref{thm:double-groth-master}]
Induct on the orthodontic sort order (Definition~\ref{defn:os}). In the base case $w = \mathrm{id}$, both $\mathscr G_{D(\mathrm{id})} = \mathscr G_{\emptyset}$ and $\mathfrak G_{\mathrm{id}}$ are equal to $1$.

Assume that $w$ is not sorted and that $\mathscr G_{D(w_\sort)} = \mathfrak G_{w_\sort}$. Corollary~\ref{cor:os1-orth} and Proposition~\ref{prop:os1-groth} imply that
\begin{align*}
\mathscr G_{D(w)} &= \left(\prod_{(a,b)\in D(w)\setminus D(w_\sort)}(x_a + y_b - x_ay_b)\right)\mathscr G_{D(w_\sort)}\\&=\left(\prod_{(a,b)\in D(w)\setminus D(w_\sort)}(x_a + y_b - x_ay_b)\right)\mathfrak G_{w_\sort}\\&=\mathfrak G_w.
\end{align*}
Now assume that $w$ is sorted and that $\mathscr G_{D(ws_{i_1}\dots s_{\alpha+1})} = \mathfrak G_{ws_{i_1}\dots s_{\alpha+1}}$. Writing $w'\colonequals ws_{i_1}\dots s_{\alpha+1}$ and $\mathcal S\colonequals \{h-\beta+1, \dots, h\}$, Theorem~\ref{thm:os2-orth} and Lemma~\ref{lem:pi-to-del} imply that
\begin{align*}
\mathscr G_{D(w)} &= \overline\pi_{i_1,h+1-\beta} \dots \overline\pi_{\alpha+1,h}\left(\prod_{s\in \mathcal S} (x_{\alpha + 1} + y_s - x_{\alpha+1}y_s)^{-1} \cdot \mathscr G_{D(w')}\right)
\\&= \overline\del_{i_1}\dots\overline\del_{\alpha+1}\left(\mathscr G_{D(w')}\right)
\\&= \overline\del_{i_1}\dots\overline\del_{\alpha+1}\left(\mathfrak G_{w'}\right)
\\&= \mathfrak G_w.\qedhere
\end{align*}
\end{proof}
\begin{cor}
\label{cor:double-schub-master}
If $D = D(w)$ is a Rothe diagram, then $\mathscr S_D(\mathbf x; \mathbf y) = \mathfrak S_w(\mathbf x; -\mathbf y)$.
\end{cor}
\begin{proof}
By definition, $\mathfrak S_w(\mathbf x; \mathbf y)$ is the lowest degree part of $\mathfrak G_w(\mathbf x; -\mathbf y)$. Remark~\ref{rem:sd-nonempty} guarantees that $\mathscr S_D(\mathbf x; \mathbf y)$ is the lowest degree part of $\mathscr G_D(\mathbf x; \mathbf y)$. The result now follows from Theorem~\ref{thm:double-groth-master}.
\end{proof}

\section{Lascoux positivity}
We prove Theorem~\ref{thm:lascoux-positivity-main} by reducing the problem to the special case where every column is a standard interval (cf.\ Lemma~\ref{lem:overlinepi-positivity} and Corollary~\ref{cor:omegas-then-pis}). This case can be checked explicitly (Proposition~\ref{prop:base-positivity}).

Recall the operator $r_{m,n}$ introduced in \cite{yu23} that acts on polynomials in $n$ variable via
\[
f(x_1, \dots, x_n)\mapsto x_1^m\dots x_n^mf(x_n^{-1},\dots, x_1^{-1}).
\]

\begin{lem}
For $i\in[n]$ and $M\subseteq[n]$, and $f\in\CC[\mathbf x]$ satisfying $\deg_{x_i}(f) \leq m$, the equalities
\begin{align}
\label{eqn:omega-spec}
\omega_i^M |_{y_j \mapsto -1} &= (x_1 - 1)^{|M|}\dots(x_i - 1)^{|M|}\\
\label{eqn:omega-int}
r_{m+1,n}\left((x_1 - 1)\dots(x_i - 1)f\right) &=x_1\dots x_{n-i}(1 - x_{n-i+1})\dots(1 - x_n) r_{m,n}(f)\\
\label{eqn:pi-spec}
\pi_{i,j}(f)|_{y_j\mapsto -1} &= \del_i((x_i- 1)(f))\\
\label{eqn:pi-int}
r_{m,n}(\del_i((x_i- 1)f)) &= \overline\pi_{n-i}(r_{m,n}(f))
\end{align}
hold.
\end{lem}
\begin{proof}
Equations~\eqref{eqn:omega-spec} and~\eqref{eqn:pi-spec} are immediate from the definition of $\omega_i^M$ and $\pi_{i,j}$. Equations~\eqref{eqn:omega-int} and~\eqref{eqn:pi-int} follow by checking monomials and extending linearly. (Equation~\eqref{eqn:pi-int} also follows from~\cite[Lem 3.4]{yu23}).
\end{proof}

Let $\varphi_i$ denote the operator 
\[
\varphi_i\colon f\mapsto x_1 \dots x_i (1 - x_{i+1})\dots(1 - x_n)f.
\]
\begin{cor}
\label{cor:rmn-via-operators}
The polynomial $r_{m,n}(\mathscr S_D(\mathbf x; -\mathbf 1))$ can be obtained from the polynomial $1\in\CC[\mathbf x]$ by repeated application of operators of the form $f\mapsto \overline\pi_i(f)$ and $\varphi_i$.
\end{cor}
\begin{proof}
By the orthodontic formula~\eqref{eqn:double-groth-master}, along with Equations~\eqref{eqn:omega-spec} and~\eqref{eqn:pi-spec}, $\mathscr S_D(\mathbf x; -\mathbf 1)$ can be obtained from the polynomial $1\in\CC[\mathbf x]$ by repeated application of operators of the form
\[
f\mapsto \del_i(x_i - 1)f \qquad\textup{ and } \qquad f\mapsto (x_1 - 1)\dots (x_i - 1)f.
\]
Using Equations~\eqref{eqn:omega-int} and~\eqref{eqn:pi-int} to commute the operator $r_{m,n}$ past the two operators above, it follows that $r_{m,n}(\mathscr S_D(\mathbf x; \mathbf {-1}))$ can be obtained by repeated application of the operators $\overline\pi_i$ and $\varphi_{n-i}$.
\end{proof}

\begin{lem}
\label{lem:overlinepi-positivity}
The equality
\[
\overline\pi_i(\mathfrak L_\alpha) = \begin{cases} \mathfrak L_\alpha &\textup{ if } \alpha_i \leq \alpha_{i+1}\\ \mathfrak L_{\alpha\cdot s_i} &\textup{ if } \alpha_i > \alpha_{i+1}\end{cases}
\]
holds; in particular, $\overline\pi_i$ preserves Lascoux positivity.
\end{lem}
\begin{proof}
When $\alpha_i \leq \alpha_{i+1}$, the definition of Lascoux polynomials guarantees that $\mathfrak L_\alpha = \overline\pi_i\mathfrak L_{\alpha\cdot s_i}$. The claim then follows from the fact that $\overline\pi_i$ is idempotent. When $\alpha_i > \alpha_{i+1}$, the claim follows from the definition of $\mathfrak L_{\alpha \cdot s_i}$.
\end{proof}

\begin{prop}
\label{prop:mult-implies-main}
Conjecture~\ref{conj:mult-pos} implies Conjecture~\ref{conj:lascoux-positivity-general}.
\end{prop}
\begin{proof}
By Corollary~\ref{cor:rmn-via-operators}, it suffices to show that the operators $\overline\pi_i$ and $\varphi_i$ preserve Lascoux positivity. This follows from Lemma~\ref{lem:overlinepi-positivity} and Conjecture~\ref{conj:mult-pos}, respectively.
\end{proof}

\begin{lem}
\label{lem:orth-of-incl}
Assume that the columns of $D$ are ordered by inclusion, and write $\mathbf K, \mathbf i, \mathbf j, \mathbf M$ for the orthodontic sequence. Then:
\begin{enumerate}
\item Every diagram appearing in the orthodontic sequence of $D$ also has all columns ordered by inclusion.
\item If $K_k\neq\emptyset$, then $i_j\neq k$ for all $j$.
\item If $M_k \neq \emptyset$ for some $k$, then $i_j \neq i_k$ for all $j > k$.
\end{enumerate}
\end{lem}
\begin{proof}
Part 1 of the lemma follows from the fact that orthodontic moves either remove a column or swap two rows; both of these moves preserve the inclusion property of the columns.

If $K_k\neq\emptyset$, then $D$ has a column equal to $[k]$. As the columns of $D$ are ordered by inclusion, every other column of $D$ is either contained in $[k]$ or is equal to $[k]\sqcup S$ for some $S$; in particular, no empty box in the $k$\textsuperscript{th} row has a box below it in its column. As orthodontic moves remove columns or move boxes up, no empty box in the $k$\textsuperscript{th} row of any diagram $D'$ appearing in the orthodontic sequence of $D$ has a box below it in its column; in particular, $k$ can never be a missing tooth of $D'$. Part 2 of the lemma follows.

Similarly, if $M_k\neq\emptyset$ for some $k$, then there is a diagram $D'$ in the orthodontic sequence of $D$ which has a column equal to $[i_k]$. Arguing as above, it follows that $i_k$ can never be a missing tooth of a diagram in the orthodontic sequence of $D'$. Part 3 of the lemma follows. 
\end{proof}

\begin{cor}
\label{cor:omegas-then-pis}
Let $D$ be a diagram whose columns are ordered by inclusion, and let $\mathbf K, \mathbf i, \mathbf j, \mathbf M$ denote the orthodontic sequence of $D$. Then,
\[
\mathscr S_D(\mathbf x; \mathbf y) = \pi_{i_1,j_1}(\dots\pi_{i_\ell,j_\ell}(\omega_1^{K_1}\dots\omega_n^{K_n}\cdot \omega_{i_1}^{M_1}\dots\omega_{i_\ell}^{M_\ell})\dots).
\]
\end{cor}
\begin{proof}
Because $\omega_i^J$ is symmetric with respect to $x_1, \dots, x_i$ and also with respect to $x_{i+1}, \dots, x_n$, it follows that $\omega_i^J$ commutes with $\pi_{i',j}$ whenever $i' \neq i$. Lemma~\ref{lem:orth-of-incl} guarantees that the $\omega_i^J$ in the orthodontic formula~\eqref{eqn:double-groth-master} can be commuted past the $\pi_{i',j}$ occurring to the right. The claim follows.
\end{proof}
Let $C_{n,k,\ell}$ be the set of $\alpha = (\alpha_1, \dots, \alpha_n)\in\NN^n$ with $\alpha_1 = \dots = \alpha_k = \ell $ and $\alpha_j \leq \ell$ for all $j$.
\begin{lem}
\label{lem:Cnkl}
Let $\alpha \in C_{n,k,\ell}$ and $i \leq k$. Then
\[
\mathfrak L_\alpha(x_1, \dots, x_n) = x_1^\ell \dots x_i^\ell \mathfrak L_{\alpha(i)}(x_{i+1},\dots, x_n), \qquad\textup{ where }\alpha(i)\colonequals (\alpha_{i+1}, \dots, \alpha_n).
\]
\end{lem}
\begin{proof}
Write $\lambda\colonequals \sort(\alpha)$ for the partition obtained by reordering $\alpha$ so that the components are weakly decreasing. Then
\[
\mathfrak L_\alpha = \overline \pi_{i_1}\dots\overline \pi_{i_r}(\mathbf x^\lambda),
\]
where $i_j > k$ for all $j$. As $\lambda_1 = \dots = \lambda_k = \ell$, 
\[
\mathfrak L_\alpha = x_1^\ell\dots x_i^\ell \overline \pi_{i_1}\dots\overline \pi_{i_\ell}(\mathbf x^{\lambda(i)}), \qquad \lambda(i)\colonequals (\lambda_{i+1}, \dots, \lambda_n)
\]
for any $i \leq k$. The claim follows from the fact that
\[
\overline \pi_{i_1}\dots\overline \pi_{i_\ell}(\mathbf x^{\lambda(i)}) = \mathfrak L_{\alpha(i)}(x_{i+1}, \dots, x_n).\qedhere
\]
\end{proof}
\begin{lem}
\label{lem:Cnkl-basis}
The set of Lascoux polynomials $\{\mathfrak L_\alpha\colon \alpha \in C_{n,k,\ell}\}$ forms a basis for the vector space $V_{k,\ell}$ of polynomials of the form
\[
x_1^\ell \dots x_k^\ell \cdot f(x_{k+1}, \dots, x_n), \qquad \deg_{x_j}(f) \leq \ell\mbox{\;for all }j>k.
\]
\end{lem}
\begin{proof}
The set of polynomials $\mathfrak L_\alpha(x_{k+1}, \dots, x_n)$ with $\alpha_i\leq \ell$ forms a basis for the set of polynomials $\CC[x_{k+1}, \dots, x_n]$ such that $\deg_{x_i}(f) \leq \ell$ for all $i>k$. Lemma~\ref{lem:Cnkl} implies that $\{\mathfrak L_\alpha\colon \alpha \in C_{n,k,\ell}\}$ spans $V_{k,\ell}$. Since Lascoux polynomials are linearly independent, the claim follows.
\end{proof}

Recall that the stable Grothendieck polynomial $G_w(x_1, \dots, x_n)$ of a permutation $w\in S_m$ is defined to be the limit $\lim_{N\to\infty} \mathfrak G_{1^N\times w}(x_1, \dots, x_n, 0, 0,\dots)$.

\begin{example}
\label{ex:g21}
For $w = 21$, the stable Grothendieck polynomial $G_{21}(x_1, \dots, x_n)$ is equal to
\[
G_{21}(x_1, \dots, x_n) = \sum_{i=1}^n(-1)^{i+1}\mathbf e_i(\mathbf x),
\]
where $\mathbf e_i$ is the $i$\textsuperscript{th} elementary symmetric polynomial. (More generally, for permutations with a unique descent, $G_w(x_1, \dots, x_n)$ can be computed e.g.\ using \cite[Thm 2.2]{lenart00}.)

In particular,
\[
(1-x_1)\dots(1-x_n) = 1 - G_{21}(x_1, \dots, x_n).\qedhere
\]
\end{example}

The key ingredient for the Lascoux positivity in Theorem~\ref{thm:lascoux-positivity-main} is the following result by Orelowitz and Yu.
\begin{thm}[{\cite[Thm 1.3]{oy23}}]
\label{thm:tianyi-hammer}
The product $\mathfrak L_\alpha(x_1, \dots, x_n) \cdot G_w(x_1, \dots, x_n)$ is a graded nonnegative sum of Lascoux polynomials:
\[
\mathfrak L_\alpha(x_1, \dots, x_n) \cdot G_w(x_1, \dots, x_n) = \sum_\beta c_\beta (-1)^{|\beta| - \ell(w) - |\alpha|} \mathfrak L_\beta(x_1, \dots, x_n), \qquad c_\beta \geq 0.
\]
\end{thm}
Orelowitz and Yu prove more in \cite[Thm 1.3]{oy23}: they express the product as a sum over certain tableaux, each of which contributes a Lascoux polynomial in the expansion. In the following proposition, we abuse notation and write $\varphi_i$ for the polynomial $x_1 \dots x_i(1 - x_{i+1})\dots(1-x_n)$.

\begin{prop}
\label{prop:base-positivity}
Assume that $\varphi_i^{a_i}\varphi_{i+1}^{a_{i+1}}\dots\varphi_n^{a_n}$ is a graded nonnegative sum of Lascoux polynomials for some $a_i, a_{i+1}, \dots, a_n \geq 0$. Then
\[
\varphi_i \cdot \varphi_i^{a_i}\varphi_{i+1}^{a_{i+1}}\dots\varphi_n^{a_n}
\]
is again a graded nonnegative sum of Lascoux polynomials.
\end{prop}
\begin{proof}
For $i = n$, the result is trivial.

Assume that $i \leq n-1$ and write $f\colonequals \varphi_i^{a_i}\varphi_{i+1}^{a_{i+1}}\dots\varphi_n^{a_n}$. Let $d \colonequals ia_i + \dots + n a_n$ denote the lowest total degree of $f$. Write $\ell\colonequals a_i + \dots + a_n$ and observe that $f\in V_{i,\ell}$. Lemma~\ref{lem:Cnkl-basis} implies that the Lascoux expansion of $f$, graded nonnegative by assumption, reads
\[
f = \sum_{\alpha\in C_{n,i,\ell}}(-1)^{|\alpha| - d}c_\alpha\mathfrak L_\alpha(x_1, \dots, x_n), \qquad c_\alpha \geq 0.
\]
By Lemma~\ref{lem:Cnkl}, the polynomial $f$ expands as
\[
f = \sum_{\alpha \in C_{n,i,\ell}} (-1)^{|\alpha| - d} c_\alpha x_1^\ell \dots x_i^\ell \cdot \mathfrak L_{\alpha(i)}(x_{i+1}, \dots, x_n), \qquad\textup{ where } \alpha(i) \colonequals (\alpha_{i+1}, \dots, \alpha_n).
\]
As $\varphi_i = x_1 \dots x_i(1 - G_{21}(x_{i+1}, \dots, x_n))$ (cf.\ Example~\ref{ex:g21}), the polynomial $\varphi_i f$ expands as
\begin{equation}
\label{eqn:varphiif-expansion}
\varphi_i f = \sum_{\alpha \in C_{n,i,\ell}}(-1)^{|\alpha| - d} c_\alpha x_1^{\ell+1}\dots x_i^{\ell+1} \cdot (1 - G_{21}(x_{i+1}, \dots, x_n)) \cdot \mathfrak L_{\alpha(i)}(x_{i+1}, \dots, x_n).
\end{equation}
Theorem~\ref{thm:tianyi-hammer} implies that
\[
G_{21}(x_{i+1}, \dots, x_n) \mathfrak L_{\alpha(i)}(x_{i+1}, \dots, x_n) = \sum_{\beta(i)} c_{\alpha(i),\beta(i)} (-1)^{|\beta(i)| - 1 - |\alpha(i)|}\mathfrak L_{\beta(i)}(x_{i+1}, \dots, x_n).
\]
As $\deg_{x_j}(\mathfrak L_{\alpha(i)})\leq \ell$ and $\deg_{x_j}(G_{21}(x_{i+1}, \dots, x_n)) = 1$ for $i+1 \leq j \leq n$, Lemma~\ref{lem:Cnkl-basis} implies that $\deg_{x_j}(\mathfrak L_{\beta(i)}) \leq \ell+1$ for $i+1 \leq j\leq n$. Writing $\beta(i)\colonequals (\beta_{i+1}, \dots, \beta_n)$, Lemma~\ref{lem:Cnkl} implies that
\[
x_1^{\ell+1} \dots x_i^{\ell+1}\mathfrak L_{\beta(i)}(x_{i+1}, \dots, x_n) = \mathfrak L_\beta(x_1, \dots, x_n), \qquad \beta\colonequals (\underbrace{\ell+1, \dots, \ell+1}_{i \textup{ many}}, \beta_{i+1}, \dots, \beta_n).
\]
It follows that each summand in Equation~\eqref{eqn:varphiif-expansion} has a graded nonnegative Lascoux expansion.
\end{proof}

\newtheorem*{thm:lascoux-positivity-main}{Theorem~\ref{thm:lascoux-positivity-main}}
\begin{thm:lascoux-positivity-main}
Let $D\subseteq[n]\times[m]$ be a diagram whose columns are ordered by inclusion. Let $\mathscr S_D(\mathbf x; \mathbf y)$ be the lowest degree part of $\mathscr G_D(\mathbf x; \mathbf y)$. Then, the polynomial
\[
x_1^m \dots x_n^m \mathscr S_D(x_n^{-1}, \dots, x_1^{-1}; -1, \dots, -1)
\]
is a graded nonnegative sum of Lascoux polynomials $\mathfrak L_\alpha(x_1, \dots, x_n)$.
\end{thm:lascoux-positivity-main}
\begin{proof}
Corollary~\ref{cor:omegas-then-pis} asserts
\[
\mathscr S_D(\mathbf x; \mathbf y) = \pi_{i_1,j_1}(\dots\pi_{i_\ell,j_\ell}(\omega_1^{K_1}\dots\omega_n^{K_n}\cdot \omega_{i_1}^{M_1}\dots\omega_{i_\ell}^{M_\ell})\dots).
\]
Equations~\eqref{eqn:omega-spec}--\eqref{eqn:pi-int}, along with the fact that the $\varphi_i$ commute with each other, imply that
\[
r_{m,n}(\mathscr S_D(\mathbf x; -\mathbf 1)) = \overline\pi_{n - i_1}\dots\overline\pi_{n - i_\ell}(\varphi_{k_1}\dots\varphi_{k_m}(1))
\]
for $k_1 \geq \dots \geq k_m$. Proposition~\ref{prop:base-positivity} implies that $\varphi_{k_1}\dots\varphi_{k_m}(1)$ is Lascoux positive, and the result follows from Lemma~\ref{lem:overlinepi-positivity}.
\end{proof}
\newtheorem*{cor:lascoux-positivity-schub}{Corollary~\ref{cor:lascoux-positivity-schub}}

\begin{cor:lascoux-positivity-schub}
Let $w\in S_n$ be a vexillary permutation and write $\mathfrak S_w(x_1, \dots, x_n; y_1, \dots, y_n)$ for the double Schubert polynomial. Then the polynomial
\[
x_1^n \dots x_n^n \mathfrak S_w(x_n^{-1}, \dots, x_1^{-1}; 1, \dots, 1)
\]
is a graded nonnegative sum of Lascoux polynomials $\mathfrak L_\alpha(x_1, \dots, x_n)$.
\end{cor:lascoux-positivity-schub}
\begin{proof}
By Corollary~\ref{cor:double-schub-master}, the double Schubert polynomial $\mathfrak S_w(\mathbf x; \mathbf 1)$ is equal to $\mathscr S_{D(w)}(\mathbf x; -\mathbf 1)$. The result follows from Theorem~\ref{thm:lascoux-positivity-main}.
\end{proof}

%\begin{example}
%	%nonvexillary
%	If $w=2143$, then
%	\begin{align*}
%		\mathfrak{S}_w(x_1,x_2,x_3;y_1,y_2,y_3)&=
%		x_1^2
%		+x_1 x_2
%		+x_1 x_3
%		-2 x_1 y_1
%		-x_1 y_2
%		-x_1 y_3
%		-x_2 y_1
%		-x_3 y_1
%		+y_1^2
%		+y_1 y_2
%		+y_1 y_3
%	\end{align*}
%	which gives
%	\begin{align*}
%		x_1^3x_2^3x_3^3 \mathfrak S_w(x_3^{-1}, x_2^{-1}, x_1^{-1}; 1, 1, 1)&=
%		3 x_1^3 x_2^3 x_3^3
%		-x_1^3 x_2^2 x_3^3 
%		-x_1^3 4 x_2^3 x_3^2 
%		+x_1^3 x_2^2 x_3^2 
%		+x_1^3 x_2^3 x_3 
%		-x_1^2 x_2^3 x_3^3 
%		+x_1^2 x_2^3 x_3^2\\
%		&=\left(\mathfrak L_{(2, 3, 2)} + \mathfrak L_{(3, 3, 1)}\right) - \left(\mathfrak L_{(2, 3, 3)} + 2\mathfrak L_{(3, 3, 2)}\right) + \mathfrak L_{(3, 3, 3)}.\qedhere
%	\end{align*}
%\end{example}

\begin{example}
	%dominant
	Replacing $w \in S_n$ with $w(n+1) \in S_{n+1}$ amounts to replacing 
	\[f\colonequals x_1^n\dots x_n^n\mathfrak S_w(x_n^{-1}, \dots, x_1^{-1};\mathbf 1)\quad \mbox{with}\quad x_1^{n+1}x_2\dots x_n x_{n+1} f(x_2,\dots,x_{n+1}).\]
	This operation preserves Lascoux positivity by Lemma~\ref{lem:Cnkl}. For example, if $w = 321$, then
	\begin{align*}
		\mathfrak{S}_w(\mathbf{x};\mathbf{y})&=\left(x_1-y_1\right) \left(x_2-y_1\right) \left(x_1-y_2\right)\\
		&=
			 x_1^2 x_2
			-x_1^2 y_1
			-x_1 x_2 y_1
			-x_1 x_2 y_2
			+x_1 y_1^2
			+x_1 y_1 y_2
			+x_2 y_1 y_2
			-y_1^2 y_2,
	\end{align*}
	which gives
	\begin{align*}
		x_1^3x_2^3x_3^3 \mathfrak S_w(x_3^{-1},x_2^{-1}, x_1^{-1}; \mathbf{1})
		&=
		x_1^3 x_2^2 x_3
		-x_1^3 x_2^3 x_3
		-x_1^3 x_2^2 x_3^2
		-x_1^3 x_2^2 x_3^2\\
		&\qquad
		+x_1^3 x_2^3 x_3^2
		+x_1^3 x_2^3 x_3^2
		+x_1^3 x_2^2 x_3^3
		-x_1^3 x_2^3 x_3^3\\
		&=\left(\mathfrak L_{321}\right) - \left(2\mathfrak L_{322} + \mathfrak L_{331}\right) + \left(\mathfrak L_{323}+\mathfrak L_{332}\right).
	\end{align*}
	Taking $w = 3214$ gives
	\begin{align*}
		x_1^4x_2^4x_3^4 x_4^4 \mathfrak S_w(x_4^{-1},x_3^{-1}, x_2^{-1}, x_1^{-1}; \mathbf{1})
		&=
		x_1^4 x_2^4 x_3^3 x_4^2
		-x_1^4 x_2^4 x_3^4 x_4^2
		-x_1^4 x_2^4 x_3^3 x_4^3
		-x_1^4 x_2^4 x_3^3 x_4^3\\
		&\qquad
		+x_1^4 x_2^4 x_3^4 x_4^3
		+x_1^4 x_2^4 x_3^4 x_4^3
		+x_1^4 x_2^4 x_3^3 x_4^4
		-x_1^4 x_2^4 x_3^4 x_4^4\\
%		&=\left(\mathfrak L_{(4, 4, 3, 2)}\right) - \left(2\mathfrak L_{(4, 4, 3, 3)} + \mathfrak L_{(4, 4, 4, 2)}\right) + \left(\mathfrak L_{(4, 4, 3, 
%			4)}+\mathfrak L_{(4, 4, 4, 3)}\right).		
		&=\left(\mathfrak L_{4432}\right) - \left(2\mathfrak L_{4433} + \mathfrak L_{4442}\right) + \left(\mathfrak L_{4434}+\mathfrak L_{4443}\right).	\qedhere
	\end{align*}
\end{example}

\end{document}